\documentclass[12pt]{amsart}
\usepackage{amssymb}
\usepackage{enumitem}
\usepackage{graphicx}

\makeatletter
\@namedef{subjclassname@2010}{%
  \textup{2010} Mathematics Subject Classification}
\makeatother

\usepackage[T1]{fontenc}
\newtheorem{theorem}{Theorem}[section]

\newtheorem{lemma}[theorem]{Lemma}

\theoremstyle{definition}

\newtheorem{remark}[theorem]{Remark}

\numberwithin{equation}{section}
\def\Z{\Bbb Z}
\def\N{\Bbb N}

\def\C{\Bbb C}
\def\l{\left}
\def\r{\right}
\def\bg{\bigg}
\def\({\bg(}
\def\){\bg)}
\def\t{\text}
\def\f{\frac}

\def\ls{\leqslant}

\def\bi{\binom}

\def\M#1#2{{#1\brack #2}_q}

\theoremstyle{plain}

\begin{document}

\baselineskip=17pt
\hbox{Colloq. Math. 158 (2019), no.\,2, 313--320.}
\medskip

\title[Two $q$-analogues of Euler's formula $\zeta(2)=\pi^2/6$]{Two $q$-analogues of Euler's formula $\zeta(2)=\pi^2/6$}

\author[Z.-W. Sun]{Zhi-Wei Sun}
\address{Department of Mathematics\\ Nanjing University\\
Nanjing 210093\\
People's Republic of China}
\email{zwsun@nju.edu.cn}

\date{}

\begin{abstract}
It is well known that $\zeta(2)=\pi^2/6$ as discovered by Euler.
In this paper we present the following two $q$-analogues of this celebrated formula:
$$\sum_{k=0}^\infty\f{q^k(1+q^{2k+1})}{(1-q^{2k+1})^2}=\prod_{n=1}^\infty\f{(1-q^{2n})^4}{(1-q^{2n-1})^4}$$
and
$$\sum_{k=0}^\infty\f{q^{2k-\lfloor(-1)^kk/2\rfloor}}{(1-q^{2k+1})^2}
=\prod_{n=1}^\infty\f{(1-q^{2n})^2(1-q^{4n})^2}{(1-q^{2n-1})^2(1-q^{4n-2})^2},$$
where $q$ is any complex number with $|q|<1$. We also give a $q$-analogue of the identity $\zeta(4)=\pi^4/90$,
and pose a problem on $q$-analogues of Euler's formula for $\zeta(2m)\ (m=3,4,\ldots)$.
\end{abstract}

\subjclass[2010]{Primary 05A30;
Secondary 11B65, 11M06, 33D05.}

\keywords{Identity, $q$-analogue, zeta function}

\maketitle

\section{Introduction}
For $n\in\N=\{0,1,2,\ldots\}$, the $q$-analogue of $n$ is defined as
$$[n]_q:=\f{1-q^n}{1-q}=\sum_{0\ls k<n}q^k\in\Z[q].$$
Note that $\lim_{q\to 1}[n]_q=n$.
For $|q|<1$, the {\it $q$-Gamma function} introduced by F. H. Jackson \cite{J} in 1905 is given by
$$\Gamma_q(x):=(1-q)^{1-x}\prod_{n=1}^\infty\f{1-q^n}{1-q^{n+x-1}}.$$
In view of the basic properties of the $q$-Gamma function (cf. \cite[pp.\,493--496]{AAR}), we have
\begin{equation}\label{1.1}\lim_{q\to1\atop |q|<1}(1-q^2)\prod_{n=1}^\infty\f{(1-q^{2n})^2}{(1-q^{2n-1})^2}
=\lim_{q\to1\atop|q|<1}\Gamma_{q^2}\l(\f12\r)^2=
\Gamma\l(\f12\r)^2=\pi,\end{equation}
which is essentially equivalent to Wallis' formula
$$\prod_{n=1}^\infty\f{4n^2}{4n^2-1}=\f{\pi}2$$
since
$$\prod_{n=1}^\infty\f{(1-q^{2n})^2}{(1-q^{2n-1})(1-q^{2n+1})}
=(1-q)\prod_{n=1}^\infty\f{(1-q^{2n})^2}{(1-q^{2n-1})^2}\ \ \t{for}\ |q|<1.$$

In light of \eqref{1.1},
$$\lim_{q\to1\atop|q|<1}(1-q)\prod_{n=1}^\infty\f{(1-q^{4n})^2}{(1-q^{4n-2})^2}
=\lim_{q\to1\atop|q|<1}\f{1-q}{1-q^4}\times \pi=\f{\pi}4$$
and hence we may view Ramanujan's formula
$$\sum_{k=0}^\infty\f{(-q)^k}{1-q^{2k+1}}=\prod_{n=1}^\infty\f{(1-q^{4n})^2}{(1-q^{4n-2})^2}\ \ \ (|q)<1)$$
(equivalent to Example (iv) in B. C. Berndt \cite[p.\,139]{B91}) as a $q$-analogue of
Leibniz's classical identity
$$\sum_{k=0}^\infty\f{(-1)^k}{2k+1}=\f{\pi}4.$$
Recently, V.J.W. Guo and J.-C. Liu \cite{GL} gave some $q$-analogues of two Ramanujan-type series for $1/\pi$.

Let $\C$ be the field of complex numbers. The {\it Riemann zeta function} is given by
$$\zeta(s):=\sum_{n=1}^\infty\f1{n^s}\quad \ \t{for}\ s\in\C\ \t{with}\ \Re(s)>1.$$
In 1734 L. Euler obtained the elegant formula
\begin{equation}\label{1.2}\zeta(2)=\f{\pi^2}6.\end{equation}
In 2011 Kh. Hessami Pilehrood and T. Hessami Pilehrood \cite{HP}
gave an interesting $q$-analogue of the known identity $3\sum_{n=1}^\infty1/(n^2\bi{2n}n)=\zeta(2),$
which states that
$$\sum_{n=1}^\infty\f{q^{n^2}(1+2q^n)}{[n]_q^2\M{2n}n}=\sum_{n=1}^\infty\f{q^n}{[n]_q^2} \quad\t{for}\ |q|<1,$$
where
$$\M{2n}n=\f{\prod_{k=1}^{2n}[k]_q}{\prod_{j=1}^n[j]_q^2}$$
is the $q$-analogue of the central binomial coefficient $\bi{2n}n$.

Euler's celebrated formula $(1.2)$ plays very important roles in modern mathematics.
Though $\sum_{n=1}^\infty q^n/{[n]_q^2}$ (with $q|<1$) is a natural $q$-analogue of $\zeta(2)$,
it seems hopeless to use it to give a $q$-analogue of \eqref{1.2}.
As nobody has given a $q$-analogue of Euler's formula \eqref{1.2} before,
we aim to present two $q$-analogues of \eqref{1.2} in this paper. 

Our main result is as follows.

\begin{theorem}\label{Th1.1} For any $q\in\C$ with $|q|<1$, we have
\begin{equation}\label{1.3}\sum_{k=0}^\infty\f{q^k(1+q^{2k+1})}{(1-q^{2k+1})^2}
=\prod_{n=1}^\infty\f{(1-q^{2n})^4}{(1-q^{2n-1})^4},\end{equation}
\begin{equation}\label{1.4}\sum_{k=0}^\infty\f{q^{2k-\lfloor(-1)^kk/2\rfloor}}{(1-q^{2k+1})^2}
=\prod_{n=1}^\infty\f{(1-q^{2n})^2(1-q^{4n})^2}{(1-q^{2n-1})^2(1-q^{4n-2})^2}.
\end{equation}
\end{theorem}

Clearly,
$$\sum_{k=0}^\infty\f1{(2k+1)^2}=\sum_{n=1}^\infty\f1{n^2}-\sum_{k=1}^\infty\f1{(2k)^2}=\f34\zeta(2)$$
and hence (\ref{1.2}) has the equivalent form
\begin{equation}\label{1.5}\sum_{k=0}^\infty\f1{(2k+1)^2}=\f{\pi^2}8.\end{equation}
Now we explain why (\ref{1.5}) follows from (\ref{1.3}) or (\ref{1.4}).
In view of (\ref{1.1}),
$$\lim_{q\to 1\atop|q|<1}(1-q^2)^2\prod_{n=1}^\infty\f{(1-q^{2n})^4}{(1-q^{2n-1})^4}=\pi^2,$$
$$\lim_{q\to1\atop|q|<1}(1-q^2)(1-q^4)\prod_{n=1}^\infty
\f{(1-q^{2n})^2(1-q^{4n})^2}{(1-q^{2n-1})^2(1-q^{4n-2})^2}=\pi^2.$$
Thus
\begin{equation}\label{1.6}\lim_{q\to1\atop|q|<1}(1-q)^2\prod_{n=1}^\infty\f{(1-q^{2n})^4}{(1-q^{2n-1})^4}
=\f{\pi^2}4\end{equation}
and
$$\lim_{q\to1\atop|q|<1}(1-q)^2\prod_{n=1}^\infty\f{(1-q^{2n})^2(1-q^{4n})^2}{(1-q^{2n-1})^2(1-q^{4n-2})^2}
=\f{\pi^2}8.$$
On the other hand,
$$\lim_{q\to 1\atop|q|<1}(1-q)^2\sum_{k=0}^\infty\f{q^k(1+q^{2k+1})}{(1-q^{2k+1})^2}
=\lim_{q\to 1\atop|q|<1}\sum_{k=0}^\infty\f{q^k(1+q^{2k+1})}{[2k+1]_q^2}=\sum_{k=0}^\infty\f2{(2k+1)^2},$$
$$\lim_{q\to 1\atop|q|<1}(1-q)^2\sum_{k=0}^\infty\f{q^{2k-\lfloor(-1)^kk/2\rfloor}}{(1-q^{2k+1})^2}
=\lim_{q\to1\atop|q|<1}\sum_{k=0}^\infty\f{q^{2k-\lfloor(-1)^kk/2\rfloor}}{[2k+1]_q^2}=\sum_{k=0}^\infty\f1{(2k+1)^2}.$$
Therefore (\ref{1.3}) and (\ref{1.4}) indeed give $q$-analogues of $(\ref{1.2})$.

We also deduce a $q$-analogue of the known formula $\zeta(4)=\pi^4/90$, which has the equivalent form
\begin{equation}\label{1.7}\sum_{k=0}^\infty\f1{(2k+1)^4}=\f{\pi^4}{96}.\end{equation}

\begin{theorem}\label{Th1.2} For any $q\in\C$ with $|q|<1$, we have
\begin{equation}\label{1.8}\sum_{k=0}^\infty\f{q^{2k}(1+4q^{2k+1}+q^{4k+2})}{(1-q^{2k+1})^4}=\prod_{n=1}^\infty\f{(1-q^{2n})^8}{(1-q^{2n-1})^8}.
\end{equation}
\end{theorem}

As is clearly seen, the left-hand side of (\ref{1.8}) times $(1-q)^4$
tends to $6\sum_{k=0}^\infty1/(2k+1)^4$ as $q\to1$. In view of (\ref{1.6}), the right-hand side of (\ref{1.8}) times $(1-q)^4$ has the limit $\pi^4/16$ as $q\to 1$. So (\ref{1.8}) implies (\ref{1.7}) and hence it gives a $q$-analogue of the formula $\zeta(4)=\pi^4/90$.

We will show Theorems \ref{Th1.1} and \ref{Th1.2} in the next section.

The Bernoulli numbers $B_0,B_1,\ldots$ are given by $B_0=1$ and
$$\sum_{k=0}^n\bi{n+1}kB_k=0\ \ \ (n=1,2,3,\ldots).$$
Euler proved (cf. \cite[pp.\,231--232]{IR}) that for each $m=1,2,3,\ldots$ we have
\begin{equation}\label{1.9}\zeta(2m)=(-1)^{m-1}\f{2^{2m-1}\pi^{2m}}{(2m)!}B_{2m}.\end{equation}
To seek for $q$-analogues of (\ref{1.9}) is our novel idea in this paper.
We don't know whether one can find a $q$-analogue of (\ref{1.9}) similar to (\ref{1.3}), (\ref{1.4}) and (\ref{1.8})
for $m=3,4,5,\ldots$, and this problem might stimulate further research.

\section{Proofs of Theorems \ref{Th1.1} and \ref{Th1.2}}

Recall that the {\it triangular numbers} are the integers
$$T_n:=\f{n(n+1)}2\ \ (n=0,1,2,\ldots).$$
As usual, we set

\begin{equation}\label{2.1}\psi(q):=\sum_{n=0}^\infty q^{T_n}\quad\t{for}\ |q|<1.\end{equation}

\begin{lemma}\label{Lem2.1} For $|q|<1$ we have
\begin{equation}\label{2.2}\psi(q)=\prod_{n=1}^\infty\f{1-q^{2n}}{1-q^{2n-1}}.\end{equation}
\end{lemma}
\begin{remark}\label{Rem2.1} This is a well-known result due to Gauss (cf. Berndt \cite[(1.3.14), p.\,11]{B06}).
\end{remark}

\begin{lemma}\label{Lem2.2} Let $n\in\N$ and
$$t_4(n):=|\{(w,x,y,z)\in\N^4:\ T_w+T_x+T_y+T_z=n\}|.$$
Then
\begin{equation}\label{2.3}t_4(n)=\sigma(2n+1),\end{equation}
where $\sigma(m)$ denotes the sum of all positive divisors of a positive integer $m$.
\end{lemma}
\begin{remark}\label{Rem2.2} This is also a known result, see  \cite[(3.6.6.), p.72]{B06}. In contrast with (\ref{2.3}), for any positive integer $n$
Jacobi showed that
$$r_4(n):=|\{(w,x,y,z)\in\Z^4:\ w^2+x^2+y^2+z^2=n\}|=8\sum_{4\nmid d\mid n}d$$
(cf. \cite[(3.3.1), p.\,59]{B06}).
\end{remark}

\begin{lemma}\label{Lem2.3} For $|q|<1$ we have
\begin{equation}\label{2.4}
\sum_{k=0}^\infty\f{q^k(1+q^{2k+1})}{(1-q^{2k+1})^2}=\sum_{n=0}^\infty\sigma(2n+1)q^{n}.
\end{equation}
\end{lemma}
\begin{proof} For each $k\in\N$, clearly
\begin{align*}\f{q^k(1+q^{2k+1})}{(1-q^{2k+1})^2}=&2q^k(1-q^{2k+1})^{-2}-q^k(1-q^{2k+1})^{-1}
\\=&2q^k\sum_{j=0}^\infty\bi{-2}j(-q^{2k+1})^j-q^k\sum_{j=0}^\infty q^{(2k+1)j}
\\=&\sum_{j=0}^\infty(2j+1)q^{(2j+1)(k+1/2)-1/2}.
\end{align*}
Thus
\begin{align*} \sum_{k=0}^\infty\f{q^k(1+q^{2k+1})}{(1-q^{2k+1})^2}
=&\sum_{k=0}^\infty\sum_{j=0}^\infty(2j+1)q^{((2j+1)(2k+1)-1)/2}
\\=&\sum_{n=0}^\infty\(\sum_{d\mid 2n+1}d\)q^{(2n+1-1)/2}
=\sum_{n=0}^\infty\sigma(2n+1)q^n.
\end{align*}
This concludes the proof.
\end{proof}

\begin{lemma}\label{Lem2.4} For each $n\in\N$, we have
\begin{equation}\label{2.5}|\{(u,v,x,y)\in\N^4:\ T_u+T_v+2T_x+2T_y=n\}|=\sum_{d\mid 4n+3}\f {d-(-1)^{(d-1)/2}}4.
\end{equation}
\end{lemma}
\begin{remark}\label{Rem2.3} This is a known result due to K. S. Williams \cite{W}.
\end{remark}

\medskip
\noindent{\it Proof of Theorem \ref{Th1.1}}. In view of Lemmas \ref{Lem2.1} and \ref{Lem2.2},
\begin{equation}\label{2.6}\prod_{n=1}^\infty\f{(1-q^{2n})^4}{(1-q^{2n-1})^4}=\psi(q)^4=\sum_{n=0}^\infty t_4(n)q^n=\sum_{n=0}^\infty\sigma(2n+1)q^n.\end{equation}
Combining this with (\ref{2.4}) we immediately obtain (\ref{1.3}).

By Lemmas \ref{Lem2.1} and \ref{Lem2.4}, we have
\begin{align*}&\prod_{n=1}^\infty\f{(1-q^{2n})^2(1-q^{4n})^2}{(1-q^{2n-1})^2(1-q^{4n-2})^2}
\\=&\psi(q)^2\psi(q^2)^2=\sum_{n=0}^\infty |\{(u,v,x,y)\in\N^4:\ T_u+T_v+2T_x+2T_y=n\}|q^n
\\=&\sum_{n=0}^\infty\sum_{d\mid 4n+3}\f{d-(-1)^{(d-1)/2}}4q^n
\\=&\sum_{j=0}^\infty\f{(4j+1)-1}4\sum_{k=0}^\infty q^{((4j+1)(4k+3)-3)/4}
\\&+\sum_{j=0}^\infty\f{(4j+3)+1}4\sum_{k=0}^\infty q^{((4j+3)(4k+1)-3)/4}
\\=&\sum_{k=0}^\infty\(q^k\sum_{j=0}^\infty jq^{j(4k+3)}+q^{3k}\sum_{j=0}^\infty(j+1)q^{j(4k+1)}\).
\end{align*}
For $|z|<1$, clearly
$$\sum_{j=0}^\infty(j+1)z^j=\sum_{j=1}^\infty jz^{j-1}=\sum_{j=0}^\infty\f d{dz}z^j=\f d{dz}(1-z)^{-1}=\f 1{(1-z)^2}.$$
Therefore
\begin{align*}&\prod_{n=1}^\infty\f{(1-q^{2n})^2(1-q^{4n})^2}{(1-q^{2n-1})^2(1-q^{4n-2})^2}
\\=&\sum_{k=0}^\infty\l(q^k\f{q^{4k+3}}{(1-q^{4k+3})^2}+\f{q^{3k}}{(1-q^{4k+1})^2}\r)
=\sum_{m=0}^\infty\f{q^{2m-\lfloor(-1)^mm/2\rfloor}}{(1-q^{2m+1})^2}.
\end{align*}
This proves (\ref{1.4}). The proof of Theorem \ref{Th1.1} is now complete. \qed
\medskip

\begin{remark}\label{Rem2.4} In light of (\ref{1.6}) and (\ref{2.6}), we have the curious formula
\begin{equation}\label{2.7}\lim_{q\to 1\atop|q|<1} (1-q)^2\sum_{n=0}^\infty\sigma(2n+1)q^n=\f{\pi^2}4.
\end{equation}
\end{remark}

\begin{lemma}\label{Lem2.5} Let $n\in\N$ and
$$t_8(n)=|\{(x_1,\ldots,x_8)\in\N^8:\ T_{x_1}+T_{x_2}+\cdots+T_{x_8}=n\}|.$$
Then
\begin{equation}\label{2.8}t_8(n)=\sum_{2\nmid d\mid n+1}\f{(n+1)^3}{d^3}.\end{equation}
\end{lemma}
\begin{remark}\label{Rem2.5} This is a result due to A. M. Legendre, see \cite [p.\,139]{B06}
or \cite[Theorem 5]{ORW}.
\end{remark}

\noindent{\it Proof of Theorem \ref{Th1.2}}. For $z\in\C$ with $|z|<1$, we have
$$\f{z}{(1-z)^4}=z\sum_{k=0}^\infty\bi{-4}k(-z)^k
=\sum_{k=0}^\infty\bi{k+3}3z^{k+1}=\sum_{k=1}^\infty\f{k(k+1)(k+2)}6z^k$$
and hence
\begin{align*}&\f{z(1+4z+z^2)}{(1-z)^4}
\\=&(1+4z+z^2)\sum_{k=1}^\infty k(k+1)(k+2)\f{z^k}6
\\=&\sum_{k=1}^\infty (k(k+1)(k+2)+4(k-1)k(k+1)+(k-2)(k-1)k)\f{z^k}6
\\=&\sum_{k=1}^\infty k^3z^k.
\end{align*}
Thus
$$\sum_{k=0}^\infty\f{q^{2k+1}(1+4q^{2k+1}+q^{4k+2})}{(1-q^{2k+1})^4}
=\sum_{k=0}^\infty\sum_{m=1}^\infty m^3q^{(2k+1)m}
=\sum_{n=1}^\infty\(\sum_{2\nmid d\mid n}\f{n^3}{d^3}\)q^n.$$
Combining this with Lemma \ref{Lem2.5} we obtain
$$\sum_{k=0}^\infty\f{q^{2k}(1+4q^{2k+1}+q^{4k+2})}{(1-q^{2k+1})^4}=\sum_{n=1}^\infty t_8(n-1)q^{n-1}=\psi(q)^8.$$
So, with the help of Lemma \ref{Lem2.1} we get the desired identity (\ref{1.8}). This completes the proof. \qed

\subsection*{Acknowledgements}
The initial version of this paper was posted to {\tt arXiv} in Feb. 2018. The work was supported by the National Natural Science Foundation of China (Grant No. 11571162) and the NSFC-RFBR Cooperation and Exchange Program (Grant No. 11811530072).

\end{document}